\documentclass[11pt,twoside]{amsart}

\usepackage[utf8]  {inputenc}%
\usepackage[T1]      {fontenc }%
\usepackage          {amsmath }%
\usepackage          {amsfonts}%
\usepackage          {amssymb }%
\usepackage          {amsthm  }%
\usepackage          {a4wide  }%
\usepackage          {url     }%
\usepackage          {tikz    }%
\usepackage[bookmarks=false,pdfborder={0 0 0.05}]{hyperref}
\usepackage[all]{xy}
\usepackage{lmodern}
\usepackage{tikz-cd}

\usepackage{enumitem}

\usepackage{booktabs}

\usepackage{amsmath, amsfonts, amssymb, amsthm,  wasysym, graphics, graphicx, xcolor, frcursive,comment,bbm}

\usepackage{etex}

\definecolor{darkblue}{rgb}{0.0,0,0.7} 
\definecolor{darkred}{rgb}{0.7,0,0} 

\usepackage{hyperref}
\usepackage[all]{xy}
\usepackage[T1]{fontenc}

\usepackage{MnSymbol}

\newcommand{\RR}{\mathbb R}
\newcommand{\ZZ}{\mathbb Z}

\DeclareMathOperator{\Red}{Red}
\DeclareMathOperator{\lin}{lin}

\newcommand{\Sa}{\widetilde{S}}

\theoremstyle{plain}
\newtheorem{theorem}{Theorem}[section]

\newtheorem{lemma}[theorem]{Lemma}
\newtheorem{corollary}[theorem]{Corollary}
\theoremstyle{definition}
\newtheorem{definition}[theorem]{Definition}
\newtheorem{example}[theorem]{Example}
\newtheorem{remark}[theorem]{Remark}

\newtheorem*{theoremA}{Theorem A}
\newtheorem*{theoremB}{Theorem B}
\newtheorem*{theoremC}{Theorem C}
\newtheorem*{theoremD}{Theorem D}

\author[P.~Wegener]{Patrick Wegener}
\address{Patrick Wegener, Philipps-Universit\"at Marburg, Germany}
\email{wegenerp@staff.uni-marburg.de}

\title[The Hurwitz Action in the Affine Symmetric Group]{The Hurwitz Action in the Affine Symmetric Group}

\date{\today}

\begin{document}

\begin{abstract}
Let $W$ be an affine Coxeter group of type $\widetilde A_n$, that is, the affine
symmetric group $\Sa_N$ with $N=n+1$, let $T$ be its set of reflections, and let $\Red_T(w)$ be the set of reduced reflection factorizations of an element $w\in W$. The braid group acts on
$\Red_T(w)$ by the Hurwitz action. For finite Coxeter groups it is known exactly when this
action is transitive, namely precisely for the parabolic quasi-Coxeter elements. We address this problem for the affine type $\widetilde A_n$ by determining all orbits of $\Red_T(w)$.
\end{abstract}

\maketitle

\section{Introduction}

A \emph{Coxeter group} $W$ comes with a distinguished generating set $S$ of involutions, the set of \emph{simple refelctions}, and with the larger set
\[
T=\{wsw^{-1}\ :\ w\in W,\ s\in S\}
\]
of all \emph{reflections}. Since $T$ generates $W$, every element $w\in W$ can be written as a product of reflections, and one may ask for the shortest such expression. Its length is the \emph{reflection length} $\ell_T(w)$, and an expression
\[
w=t_1t_2\cdots t_m\quad (t_i\in T,~ m=\ell_T(w))
\]
is called a \emph{reduced reflection factorization} of $w$. We collect all of them in the set $\Red_T(w)$, whose elements we regard as tuples $(t_1,\dots,t_m)$.

There is a natural way of passing from one reduced factorization to another. If $w=t_1\cdots t_m$, then also
\[
w=t_1 \cdots t_{i-1} \, (t_it_{i+1}t_i^{-1})\,t_i\,t_{i+2}\cdots t_m,
\]
So this operation produces a new tuple in $\Red_T(w)$ and more generally defines an action of the braid group $\mathcal{B}_m$ on $\Red_T(w)$, called the \emph{Hurwitz action}. The precise definition is given in Section~\ref{sec:hurwitz}.

\medskip
\noindent \textbf{Question:} For which $w$ does the Hurwitz action have only one orbit? If there is more than one orbit, how many are there?

\medskip
It has been first shown by Bessis \cite{Bes03} for finite Coxeter groups and by Igusa and Schiffler \cite{IS10} for not-necessarily-finite Coxeter groups, that the Hurwitz action for $w$ has precisely one orbit if $w$ is a Coxeter element. Later, the first question has been completely answered for finite Coxeter groups. Call $w$ a \emph{quasi-Coxeter element} if some reduced reflection factorization of $w$ generates the whole group $W$, and call it a \emph{parabolic quasi-Coxeter element} if some reduced reflection factorization generates a parabolic subgroup. By a theorem of Baumeister, Gobet, Roberts and Wegener \cite[Theorem~1.1]{BGRW17}, the Hurwitz action on $\Red_T(w)$ is transitive if and only if $w$ is a parabolic quasi-Coxeter element.

For affine Coxeter groups both questions remain open in the general. Transitivity is
known to hold for parabolic quasi-Coxeter elements \cite{Weg20,WY23}. But in general there are elements for which the Hurwitz action is transitive although the element is not a parabolic quasi-Coxeter element \cite[Ex.~5.7]{HK16}. We address both questions for the affine Coxeter group of type $\widetilde A_n$. We can realize this group as the \emph{affine symmetric group}, that is
\[
\Sa_N=S_N\ltimes Q,
\]
where $N = n+1$ und $Q$ is the \emph{root lattice}
\[
Q=\Big\{\lambda\in\ZZ^N:\ \sum_{i=1}^N\lambda_i=0\Big\}
\]
We regard the elements of $\Sa_N$ as maps $x\mapsto u(x)+\mu$ of $\RR^N$ with $u \in S_N$ a permutation of the coordinates and $\mu =(\mu_1,\dots,\mu_N) \in Q$. Thus we write an element $w\in\Sa_N$ as $w=t_\mu u$, meaning $w(x)=u(x)+\mu$. The reflections of $\Sa_N$ are the
reflections in the affine hyperplanes $x_i-x_j=k$. We denote the reflection in this hyperplane by $s_{\gamma, k}$.

Before stating the results we must introduce some notation. Let $w=t_\mu u \in \Sa_N$, and let $C_1, C_2, \dots, C_c$ be the cycles of the permutation $u$, viewed as subsets of the index set $\{1,\dots,N\}$. Fixed points of $u$ are regarded as cycles of length one, so the $C_t$ partition $\{1,\dots,N\}$ and $c$ is the total number of cycles.

\begin{definition} \label{def:winding}
The \emph{winding number} of $w$ along the cycle $C_t$ is
\[
\omega_t=\omega_t(w)=\sum_{i\in C_t}\mu_i\in\ZZ,
\]
that is, the sum of those coordinates of the translation vector $\mu$ whose indices lie in the
cycle $C_t$. Since the $C_t$ partition $\{1,\dots,N\}$ and $\mu$ has total coordinate sum $0$.
The winding numbers satisfy
\begin{equation} \label{eq:sumzero}
\omega_1+\dots+\omega_c=0.
\end{equation}
\end{definition}

\begin{definition} \label{def:balanced}
A partition $\pi$ of the set of cycles $\{C_1,\dots,C_c\}$ is called \emph{balanced} if for
every part $B\in\pi$ we have $\sum_{C_t\in B}\omega_t=0$. We write $q=q(w)$ for the largest possible number of parts of a balanced partition, and $\mathcal P(w)$ for the set of balanced partitions having exactly $q$ parts. The partitions in $\mathcal P(w)$ are called \emph{maximal balanced partitions}. For a part $B$ we let $\#B$ denote the number of cycles in $B$ and put
\[
g_B=\gcd\{\,\omega_t\ :\ C_t\in B\,\}.
\]
\end{definition}

Note that the partition into a single part is always balanced, so $q\ge1$ and $\mathcal P(w)$ is never empty. The reason balanced partitions appear is as follows: a reduced reflection factorization splits into groups of reflections that do not interact, and each such group must multiply to an element of a smaller affine symmetric group. This forces the corresponding
winding numbers to cancel.

We show the following results about reflection factorizations and in particular about the Hurwitz action in affine symmetric groups. 

\begin{theoremA}[Theorem~\ref{thm:length}]
For every $w = t_{\mu} u \in\Sa_N$ we have
\[
\ell_T(w)=N+c-2q,
\]
where $c$ is the number of cycles of $u$ and $q=q(w)$.
\end{theoremA}

Note that this result has already been obtained by Lewis, McCammond, Petersen and Schwer \cite[Thm.~4.25]{LMPS19}. We nevertheless provide an independent proof here. In particular, the proof of the upper bound constructs an explicit reduced factorization which is essential for the rest of this paper.

\begin{theoremB}[Theorem~\ref{thm:main}]
Let $w\in\Sa_N$ and let $F,F'\in\Red_T(w)$ be two reduced reflection factorizations. Denote by
$\langle F\rangle$ the subgroup of $\Sa_N$ generated by the entries of $F$. Then
\[
F\ \text{and}\ F'\ \text{lie in the same Hurwitz orbit}
\qquad\Longleftrightarrow\qquad
\langle F\rangle=\langle F'\rangle .
\]
\end{theoremB}

The ``\emph{only if}'' implication holds in any Coxeter group. The corresponding result (plus some extra condition which is not relevant in type $A$) for finite Coxeter groups has been conjectured by Lewis and proven by Douvropoulos and Lewis for parabolic quasi-Coxeter elements \cite{DL24}. 

\begin{theoremC}[Theorem~\ref{thm:count}]
For every $w\in\Sa_N$ the number of Hurwitz orbits on $\Red_T(w)$ is given by
\[
\sum_{\pi\in\mathcal P(w)}\ \prod_{B\in\pi}g_B^{\,\#B-1},
\]
where a part $B$ with $\#B=1$ contributes the factor $1$.
\end{theoremC}

\begin{theoremD}[Corollary~\ref{cor:criterion}]
The Hurwitz action on $\Red_T(w)$ is transitive if and only if the following two conditions hold:
\begin{enumerate}[label=\textup{(\alph*)}]
\item there is exactly one maximal balanced partition, i.e.\ $\#\mathcal P(w)=1$; and
\item $g_B=1$ for every part $B$ of that partition with $\#B\ge2$.
\end{enumerate}
\end{theoremD}

\medskip
We want to emphasise that Lewis and Wang \cite{LW22} have obtained very similar results for complex reflection groups and the approach in this paper is strongly influenced by the approach in \cite{LW22}.

\medskip
\textbf{Outline of the paper.} Section~\ref{ch:prelim} collects the necessary background on Coxeter groups, reflection length, the Hurwitz action, and a detailed description of the affine symmetric group and its reflections. Section~\ref{ch:winding} develops winding numbers and balanced partitions and proves the block decomposition of a reduced factorization. In Section~\ref{ch:sub} we investigate those reflection subgroups of $\Sa_N$ which are important for us. In Section~\ref{ch:length} we prove Theorem~A. Section~\ref{ch:main} contains the normal form and the proofs of Theorems~B,~C and~D. 

\medskip
\textbf{Disclosure of AI tools.} Generative AI tools, namely Claude Opus 5, were used during the exploratory and preparatory stages of this work. The author directed the mathematical strategy and proofs. Opus 5 assisted in several aspects like the formulation of auxiliary results and the elaboration of technical details, as well as consistency checks, and \LaTeX{} preparation. All AI-generated suggestions were independently reviewed and verified by the author, who assume full responsibility for the mathematical arguments, results, and final content of the paper.

\section{Preliminaries}\label{ch:prelim}

This section fixes notation and recalls the notions we need.

\subsection{Coxeter systems and reflection length}

\begin{definition}
A \emph{Coxeter system} is a pair $(W,S)$ consisting of a group $W$ and a finite generating set
$S\subseteq W$ such that $W$ has the presentation
\[
W=\big\langle\, S \ \big|\ (st)^{m(s,t)}=1 \ \text{for all}\ s,t\in S\ \text{with}\ m(s,t)<\infty \,\big\rangle,
\]
where $m(s,s)=1$ and $m(s,t)=m(t,s)\in\{2,3,\dots,\infty\}$ for $s\ne t$. The group $W$ is a
\emph{Coxeter group}, and $|S|$ is its \emph{rank}. The set of \emph{reflections} of $(W,S)$ is
$T=\{wsw^{-1}: w\in W,\ s\in S\}$.
\end{definition}

Standard references are Bourbaki \cite{Bou68} and Humphreys \cite{Hum90}. Note that $T$ is by
construction closed under conjugation, and that $S\subseteq T$, so $T$ generates $W$. Recall that for $w \in W$ we set
\[
\Red_T(w)=\{(t_1,\dots,t_m)\in T^m\ :\ t_1\cdots t_m=w,\ m=\ell_T(w)\}.
\]

\begin{example}\label{ex:S3}
Let $W=S_3$ with $S=\{(1\,2),(2\,3)\}$. Then $T=\{(1\,2),(1\,3),(2\,3)\}$ consists of all three
transpositions. The element $w=(1\,3\,2)$ is not a transposition, but is a product of two,
so $\ell_T(w)=2$, and one checks directly that
\[
\Red_T(w)=\big\{\,\big((1\,2),(1\,3)\big),\ \big((1\,3),(2\,3)\big),\ \big((2\,3),(1\,2)\big)\,\big\}.
\]
\end{example}

\subsection{The Hurwitz action}\label{sec:hurwitz}

\begin{definition}
Let $\mathcal{B}_m$ be the braid group on $m$ strands, with standard generators
$\sigma_1,\dots,\sigma_{m-1}$. The \emph{Hurwitz action} of $\mathcal{B}_m$ on $\Red_T(w)$, where
$m=\ell_T(w)$, is defined on generators by
\[
\sigma_i\cdot(t_1,\dots,t_m)=(t_1,\dots,t_{i-1},\ t_it_{i+1}t_i^{-1},\ t_i,\ t_{i+2},\dots,t_m),
\]
\[
\sigma_i^{-1}\cdot(t_1,\dots,t_m)=(t_1,\dots,t_{i-1},\ t_{i+1},\ t_{i+1}^{-1}t_it_{i+1},\ t_{i+2},\dots,t_m).
\]
For $F , F' \in \Red_T(w)$ we write $F\sim F'$ if $F$ and $F'$ lie in the same orbit, and call them \emph{Hurwitz equivalent}. When applying a braid $\sigma$ to $F \in \Red_T(w)$, we will often simply refer to this as a \emph{Hurwitz move}. Note that the action of $\mathcal{B}_m$ on $\Red_T(w)$ extends to an action on $T^m$. For $F=(t_1,\dots,t_m)$ we write $\langle F\rangle=\langle t_1,\dots,t_m\rangle\le W$.
\end{definition}

That the braid relations are satisfied is a straightforward computation (see \cite[\S1]{BGRW17}). 

\begin{lemma}\label{lem:subgroupinv}
If $F\sim F'$ then $\langle F\rangle=\langle F'\rangle$.
\end{lemma}

So the subgroup $\langle F\rangle$ is an invariant of the Hurwitz orbit. two factorizations generating different subgroups cannot be Hurwitz equivalent. Theorem~B will say that in $\Sa_N$ this single invariant already distinguishes all orbits.

\begin{example}
Continuing Example~\ref{ex:S3}, apply $\sigma_1$ to $F=\big((1\,2),(1\,3)\big)$:
\[
\sigma_1\cdot F=\big((1\,2)(1\,3)(1\,2),\ (1\,2)\big)=\big((2\,3),(1\,2)\big),
\]
and applying $\sigma_1$ once more gives $\big((1\,3),(2\,3)\big)$. So all three elements of
$\Red_T\big((1\,3\,2)\big)$ lie in one orbit.
\end{example}

\begin{remark} \label{rem:HuwritzBlocks}
A useful reformulation of the moves, which we shall use in Section~\ref{ch:main}, is that whole consecutive blocks may be moved past each other. If $F=(A,B)$ is the concatenation of two consecutive subtuples $A$ and $B$ with products $a$ and $b$, then a suitable product of the
$\sigma_i$ transforms $F$ into
\begin{equation}\label{eq:blockmove}
\big(\,aBa^{-1},\ A\,\big),
\end{equation}
where $aBa^{-1}$ means the tuple $B$ with every entry conjugated by $a$. Indeed, moving the first entry of $A$ through all of $B$ conjugates each entry of $B$ by it and places it behind $B$; repeating for the remaining entries of $A$, in order, yields \eqref{eq:blockmove}.
\end{remark}

\subsection{The affine symmetric group}\label{sec:affsym}
Let $n\ge1$ and put $N=n+1$. Let $e_1,\dots,e_N$ be the
standard basis of $\RR^N$ with the standard scalar product $(\cdot \mid \cdot)$, so
$(e_i \mid e_j)=\delta_{ij}$.

\begin{definition}
The \emph{root system of type $A_{N-1}$} is
\[
\Phi=\{e_i-e_j\ :\ 1\le i,j\le N,\ i\ne j\},
\]
and its \emph{root lattice} is
\[
Q=\ZZ\Phi=\Big\{\lambda\in\ZZ^N\ :\ \sum_{i=1}^N\lambda_i=0\Big\}.
\]
\end{definition}

\begin{definition}
The \emph{affine symmetric group} is the group of affine transformations of $\RR^N$
\[
\Sa_N=\{\,x\mapsto u(x)+\mu\ :\ u\in S_N,\ \mu\in Q\,\}=S_N\ltimes Q,
\]
where $S_N$ acts on $\RR^N$ by permuting coordinates, $u(e_i)=e_{u(i)}$. We write $t_\mu$ for
the translation $x\mapsto x+\mu$ and $w=t_\mu u$ for the element $x\mapsto u(x)+\mu$; thus
$\mu\in Q$ is the \emph{translation part} and $u$ the \emph{linear part} of $w$. The map
\[
\lin:\Sa_N\longrightarrow S_N,\qquad t_\mu u\mapsto u,
\]
is a surjective group homomorphism whose kernel is the \emph{translation subgroup}
\[
\mathcal T:=\{t_\mu\ :\ \mu\in Q\}\ \le\ \Sa_N .
\]
\end{definition}

We have an isomorphism of abelian groups
\[
\iota:Q\ \xrightarrow{\ \sim\ }\ \mathcal T,\qquad \iota(\mu)=t_\mu ,
\]
which satisfies $\iota(\mu+\nu)=t_{\mu+\nu}=t_\mu t_\nu=\iota(\mu)\iota(\nu)$.
Being the kernel of $\lin$, the subgroup $\mathcal T$ is normal in $\Sa_N$; and it is abelian, so
translations commute with one another.

The decomposition $w=t_\mu u$ is unique: $u$ is the linear part of the affine map $w$ and
$\mu=w(0)$. We write $\mu=(\mu_1,\dots,\mu_N)$ for the coordinates of $\mu$ in the standard
basis. By definition of $Q$ they satisfy $\sum_i\mu_i=0$. Multiplication is given by
\[
(t_\mu u)(t_\nu v)=t_{\mu+u(\nu)}\,(uv).
\]

\begin{remark}
The group $\Sa_N$ is the affine Coxeter group of type $\widetilde A_{N-1}$: it is a Coxeter group of rank $N$ with Coxeter generators the reflections in the walls of a fundamental alcove, see \cite[Chapter~4]{Hum90}. We shall not need the explicit generating set $S$, only the reflection set $T$, which we determine next.
\end{remark}

\subsection{Reflections, directions and levels}

\begin{definition}
For a root $\alpha\in\Phi$ and $k\in\ZZ$ let $H_{\alpha,k}=\{v\in\RR^N\ :\ (v \mid \alpha)=k\}$ be an affine hyperplane, and let $s_{\alpha,k}$ be the orthogonal reflection in $H_{\alpha,k}$, that is
\[
s_{\alpha,k}(v)=v-\big((v\mid\alpha)-k\big)\,\alpha .
\]
We call $\alpha$ the \emph{direction} and $k$ the \emph{level} of $s_{\alpha,k}$. In particular, we put
\[
s_{\alpha}(v) :=  s_{\alpha,0}(v) = v -(v \mid \alpha) \alpha.
\]
\end{definition}

The reflection $s_{\alpha,k}$ fixes $H_{\alpha,k}$ pointwise and is an involution. Since $s_{\alpha,k}=s_{-\alpha,-k}$, whenever we speak of ``the direction'' we mean it up to sign, or we fix one of the two roots.

\begin{lemma}[{\cite[Chapter~4.1]{Hum90}}] \label{lem:reflset}
The set of reflections of the Coxeter system $\Sa_N$ is
\[
T=\{s_{\alpha,k}\ :\ \alpha\in\Phi,\ k\in\ZZ\},
\]
and $s_{\alpha,k}=t_{k\alpha}s_\alpha$ where $s_\alpha=s_{\alpha,0}$. For $\alpha=e_i-e_j$ the
linear part is the transposition $\lin(s_{\alpha,k})=(i\,j)$.
\end{lemma}

The following computation of conjugates is used frequently throughout the rest of this paper and is verified by direct calculations.

\begin{lemma}\label{lem:conj}
Let $x\in\Sa_N$ have linear part $\bar x$ and translation part $\nu$, so $x(v)=\bar x(v)+\nu$.
Then for all $\beta\in\Phi$ and $l\in\ZZ$,
\[
x\,s_{\beta,l}\,x^{-1}=s_{\bar x\beta,\ l+(\nu|\bar x\beta)} .
\]
In particular:
\begin{enumerate}[label=\textup{(\roman*)}]
\item 
$\displaystyle t_\nu s_{\beta,l}t_\nu^{-1}=s_{\beta,\ l+(\nu|\beta)}$;
\item 
$\displaystyle s_{\alpha,k}s_{\beta,l}s_{\alpha,k}=s_{\gamma,m}$ with $\gamma=s_{\alpha}(\beta) = \beta-(\beta|\alpha)\alpha$ and $m=l-(\beta|\alpha)k$.
\end{enumerate}
\end{lemma}

\medskip
\begin{lemma}\label{lem:pairs}
Let $\alpha\in\Phi$ and $a,b\in\ZZ$. Then $\displaystyle s_{\alpha,a}\,s_{\alpha,b}=t_{(a-b)\alpha}$.
\end{lemma}

\section{Winding numbers and the block decomposition}\label{ch:winding}

\subsection{Winding numbers}

Throughout, $w=t_\mu u\in\Sa_N$ with $u\in S_N$ and $\mu\in Q$, and $C_1,\dots,C_c$ denote the
cycles of $u$, regarded as subsets of $\{1,\dots,N\}$ with fixed points counting as cycles of
length one. Thus $\{1,\dots,N\}=C_1\sqcup\dots\sqcup C_c$.

Recall the defintion of a winding number from Definition \ref{def:winding}.

\begin{example}\label{ex:winding}
Let $N=4$.
\begin{enumerate}[label=(\arabic*)]
\item For a translation $w=t_\lambda$ the linear part is the identity, whose cycles are the four singletons $\{1\},\{2\},\{3\},\{4\}$. So $c=4$ and $\omega_i=\lambda_i$, thus for translations the winding numbers are simply the coordinates of the translation vector.
\item Let $u=(3\,4)$ and $\mu=(1,-1,0,0)$, so $w=t_\mu u$. The cycles are
$\{1\},\{2\},\{3,4\}$, hence $c=3$ and $(\omega_1,\omega_2,\omega_3)=(1,-1,0)$.
\item Let $u=(1\,2\,3\,4)$ and $\mu=(1,0,0,-1)$. There is one cycle, $c=1$, and
$\omega_1=0$ in accordance to \eqref{eq:sumzero}.
\end{enumerate}
\end{example}

Winding numbers are an important invariant under conjugation. While the translation part $\mu$ is not invariant under conjugation, the winding numbers are.

\begin{lemma}\label{lem:winding}
Let $w=t_\mu u\in\Sa_N$ with cycles $C_1,\dots,C_c$ of $u$.
\begin{enumerate}[label=\textup{(\roman*)}]
\item For $\nu\in\ZZ^N$ we have $t_\nu\,w\,t_\nu^{-1}=t_{\mu+\nu-u(\nu)}\,u$, and
\[
\operatorname{Im}(1-u)=\Big\{x\in\RR^N\ :\ \sum_{i\in C_t}x_i=0\ \text{ for all }t\Big\}.
\]
Consequently the winding numbers $\omega_t$ along $C_t$ of $w$ do not change when $w$ is conjugated by a translation.
\item $w$ has finite order if and only if $\omega_t=0$ for all $t$.
\end{enumerate}
\end{lemma}

Before proving this statement, we consider the group $\Sa_N^{+}:=S_N\ltimes\ZZ^N$. Let $\nu \in \ZZ^N$ and $t_{\mu}u \in \Sa_N$. Then $\nu-u(\nu) = (1-u)(\nu) \in Q$ because $\sum_i \nu_i = \sum_i u(\nu)_i$. Hence $\mu + \nu -u(\nu) \in Q$ and $t_{\mu + \nu -u(\nu)} \in \Sa_N$. Therefore
\[ 
t_{\nu} (t_{\mu} u) t_{\nu}^{-1} = t_{\mu + \nu -u(\nu)} u \in \Sa_n
\]
which shows that $\Sa_N$ is normal in $\Sa_N^{+}$ and conjugation by $t_{\nu}$ with $\nu \in \ZZ^N$ is an automorphismus of $\Sa_N$ mapping reflections to reflections by Lemma~\ref{lem:conj}(i).

Furthermore, for a subset $C\subseteq\{1,\dots,N\}$ we put
\[
\RR^{C}=\operatorname{span}_\RR\{e_i:i\in C\}=\{x\in\RR^N: x_i=0\ \text{for all }i\notin C\}.
\]
Analogously, we define $\ZZ^C$.

\begin{proof}[Proof of Lemma \ref{lem:winding}]
(i) For the description of $\operatorname{Im}(1-u)$ we decompose $\RR^N$ along the cycles. Since $\{1,\dots,N\}=C_1\sqcup\dots\sqcup C_c$ we get $\RR^N=\bigoplus_{t=1}^{c}\RR^{C_t}$, and $u$ preserves each summand, because $u$ maps $C_t$ to itself and hence permutes the basis vectors $e_i$, $i\in C_t$, among themselves.

We claim that for $\sigma=u|_{C_t}$ of $\RR^{C_t}$ one has $\operatorname{Im}(1-\sigma)=\{x\in\RR^{C_t}:\sum_{i\in C_t}x_i=0\}$. The inclusion ``$\subseteq$'' holds because $\sum_{i\in C_t}((1-\sigma)(x))_i = \sum_{i\in C_t}(x-\sigma(x))_i=0$. For equality it suffices to compare dimensions. The right-hand side is a hyperplane and has dimension
$|C_t|-1$. For the dimension of the left-hand side observe that
\[ 
\ker(1-\sigma) = \{ (x_1, \ldots, x_N) \in \RR^N : x_i = 0 ~\text{for } i \notin C_t, ~x_i = x_j ~\text{for all }i,j \in C_t \}.
\]
Hence $\ker(1-\sigma)$ is a line and by the rank-nullity theorem we conclude $\dim\operatorname{Im}(1-\sigma)=|C_t|-1$. As $C_1, \ldots, C_c$ are the cycles of $u$, we obtain the claimed description of $\operatorname{Im}(1-u)$.

For $x\in\operatorname{Im}(1-u)$ the sums $\sum_{i\in C_t}x_i$ vanish, so replacing
$\mu$ by $\mu+x$ does not change any $\omega_t$, which is the second assertion.

(ii) The element $w$ is of finite order if and only if a point $x$ is fixed by $w$. The latter being equivalent to $u(x)+\mu=x$, i.e.\ $(1-u)x=\mu$. By (i) this
equation is solvable over $\RR$ precisely when all cycle sums of $\mu$ vanish, i.e.\ when all
$\omega_t=0$.
\end{proof}

\begin{remark}\label{rem:converse}
Two elements with the same linear part and the same winding numbers are in general not conjugate by a translation of $\Sa_N$. For instance let $N=3$ and $u=(1\,2\,3)$. There is a single cycle, so $\omega_1=0$ for every choice of $\mu$. But $u$ and $t_{e_1-e_2}u$ are not conjugate by any
translation of $\Sa_3$.
\end{remark}

\medskip
Recall Definition \ref{def:balanced} of a balanced partition.

\begin{lemma}\label{lem:nonzero}
Let $\pi\in\mathcal P(w)$ and let $B\in\pi$ with $\#B\ge2$. Then $\omega_t\ne0$ for every
$C_t\in B$. In particular $g_B\ge1$.
\end{lemma}

\begin{proof}
Suppose $\omega_t=0$ for some $C_t\in B$ with $\#B\ge2$. Replace the part $B$ by the two parts
$\{C_t\}$ and $B\setminus\{C_t\}$. The first has winding sum $\omega_t=0$, and the second has
winding sum $\big(\sum_{C_s\in B}\omega_s\big)-\omega_t=0-0=0$. Both parts are nonempty, since
$\#B\ge2$. So we obtain a balanced partition with $q+1$ parts, contradicting the maximality
of $q$.
\end{proof}

\begin{example}\label{ex:balanced}
We compute $\mathcal P$ in three cases with $N=4$.
\begin{enumerate}[label=(\arabic*)]
\item $\lambda=(2,-1,-1,0)$, $w=t_\lambda$: the cycles are the singletons with winding numbers $2,-1,-1,0$. A part containing $\{1\}$ must contain both $\{2\}$ and $\{3\}$; and $\{4\}$ has winding number $0$, so it forms a part by itself. Hence the unique maximal balanced partition is
$\big\{\{1\},\{2\},\{3\}\big\}\cup\big\{\{4\}\big\}$ with $q=2$, and the $\gcd$ of the first
part is $\gcd(2,-1,-1)=1$.
\item $\lambda=(1,1,-1,-1)$: here we have $q=2$, and there are two maximal balanced partitions $\big\{\{1\},\{3\}\big\},\big\{\{2\},\{4\}\big\}$ and $\big\{\{1\},\{4\}\big\},\big\{\{2\},\{3\}\big\}$.
\item $\lambda=(2,1,-1,-2)$: the unique maximal balanced partition is $\big\{\{1\},\{4\}\big\}$, $\big\{\{2\},\{3\}\big\}$, so $q=2$ and the $\gcd$ of the first part is $\gcd(2,-2)=2$.
\end{enumerate}
\end{example}

\subsection{The block decomposition of a factorization}

We aim to show that a reduced factorization decomposes into independent pieces, indexed by a balanced partition, and that the Hurwitz action respects this decomposition.

\begin{definition}\label{def:DirectionGraph}
Let $F=(t_1,\dots,t_m)$ be a tuple of reflections of $\Sa_N$. The \emph{direction graph} $G(F)$ is the graph with vertex set $\{1,\dots,N\}$ and for each reflection $t_r$ with direction $e_i-e_j$, an edge between $i$ and $j$. The \emph{blocks} of $F$ are the vertex sets of the connected components of $G(F)$ (isolated vertices form blocks of size one).
\end{definition}

Observe that two reflections whose directions involve disjoint index pairs commute as they act on complementary sets of coordinates. Hence entries of $F$ belonging to different blocks commute, and we may rearrange the entries of $F$ by blocks without changing the product. Let $F=(t_1,\dots,t_m)$ be a tuple of reflections of $\Sa_N$ with blocks $B_1, \dots, B_r$. Note that we can rearrange the entries of $F$ by using the Hurwitz action such that the entries are grouped by their blocks. More precisely, there is $(t_{i_1}, \dots, t_{i_m}) \sim (t_1,\dots,t_m)$ and $k_1 < k_2 <\dots < k_r$ such that the directions of $t_{i_1}, \dots, t_{i_{k_1}}$ correspond to edges between vertices in $B_1$ and $i_1 < \dots < i_{k_1}$, the directions of $t_{i_{k_1 +1}}, \dots, t_{i_{k_2}}$ correspond to edges between vertices in $B_2$ and $i_{k_1 +1} < \dots < i_{k_2}$, and so on.

For a subset $B\subseteq\{1,\dots,N\}$ let $S_B$ be the group of permutations of $B$,
$Q(B)=\{\lambda\in\ZZ^B:\sum_{i\in B}\lambda_i=0\}$, and $\Sa_B=S_B\ltimes Q(B)$ the
corresponding affine symmetric group, viewed inside $\Sa_N$ as acting trivially on the
coordinates outside $B$.

\begin{lemma}\label{lem:blocks}
Let $w\in\Sa_N$, let $F\in\Red_T(w)$ and let $B_1,\dots,B_r$ be the blocks of $F$. Then:
\begin{enumerate}[label=\textup{(\roman*)}]
\item each block $B_s$ is a union of cycles of $u=\lin(w)$;
\item grouping the entries of $F$ by blocks gives a factorization $w=w_1w_2\cdots w_r$ with
$w_s\in\Sa_{B_s}$, the corresponding subtuples $F_s$ are reduced reflection factorizations of $w_s$ and $\ell_T(w)=\sum_s\ell_T(w_s)$;
\item the partition of the cycles induced by $B_1,\dots,B_r$ is balanced and consequently $r\le q(w)$;
\item two reduced factorizations of $w$ with the same blocks are Hurwitz equivalent if and only
if their subtuples are Hurwitz equivalent blockwise, and the number of Hurwitz orbits with a
fixed block partition is the product over the blocks of the numbers of Hurwitz orbits on
$\Red_{T\cap\Sa_{B_s}}(w_s)$.
\end{enumerate}
\end{lemma}

\begin{proof}
Write $F=(t_1,\dots,t_m)$ and assume that the entries are already grouped by blocks. Let $F_s$ be the subtuple of entries whose direction is an edge inside $B_s$, and let $w_s$ be its product. Then $w=w_1\cdots w_r$, and $w_s\in\Sa_{B_s}$ because all its factors are.

(i) Since $\lin(w)=\lin(w_1)\cdots\lin(w_r)$ and $\lin(w_s)\in S_{B_s}$ with the $B_s$ disjoint,
$u$ preserves each $B_s$. Hence $B_s$ is a union of cycles.

(ii) If some $F_s$ were not reduced, replacing it by a shorter factorization of $w_s$ would give
a shorter factorization of $w$, contradicting $F\in\Red_T(w)$. Hence each $F_s$ is reduced and
the lengths add up.

(iii) By (i) each $B_s$ is a union of cycles, so the $B_s$ induce a partition of
$\{C_1,\dots,C_c\}$. Fix $s$ and let $\mu^{(s)}$ be the translation part of $w_s$. Since
$w_s\in\Sa_{B_s}$ we have $\mu^{(s)}\in Q(B_s)$, i.e.\ $\sum_{i\in B_s}\mu^{(s)}_i=0$.
Now the translation part of $w=w_1\cdots w_r$ is $\mu=\sum_s\mu^{(s)}$ (the linear parts act
trivially outside their own blocks), and $\mu^{(s)}$ is supported on
$B_s$, that is $(\mu^{(s)})_i = 0$ for all $i \notin B_s$. Hence
\[
\sum_{C_t\subseteq B_s}\omega_t=\sum_{i\in B_s}\mu_i=\sum_{i\in B_s}\mu^{(s)}_i=0,
\]
which shows that the induced partition is balanced. Since a balanced
partition has at most $q$ parts, $r\le q$.

(iv) A Hurwitz move applied to two neighbouring entries lying in different blocks simply interchanges them, because they commute. Hence the orbit of $F$ consists exactly of all interleavings of tuples obtained from the $F_s$ by moves inside the blocks. Note also that the blocks cannot change within an orbit. This is Lemma~\ref{lem:blocksgroup} below, which shows that the blocks are determined by the subgroup $\langle F\rangle$, and that subgroup is constant on the orbit by Lemma~\ref{lem:subgroupinv}.
\end{proof}

\begin{lemma}\label{lem:blocksgroup}
Let $F$ be any tuple of reflections of $\Sa_N$ and let $G(\langle F\rangle)$ be the graph on
$\{1,\dots,N\}$ whose edges are the directions of \emph{all} reflections contained in
$\langle F\rangle$. Then $G(F)$ and $G(\langle F\rangle)$ have the same connected components. In
particular the blocks of $F$ can be read off from the subgroup $\langle F\rangle$ alone, and are
therefore constant on Hurwitz orbits.
\end{lemma}

\begin{proof}
Since the entries of $F$ lie in $\langle F\rangle$, every edge of $G(F)$ is an edge of $G(\langle F\rangle)$, so each component of $G(F)$ is contained in a component of $G(\langle F\rangle)$. It remains to show that the reflections in $\langle F\rangle \setminus F$ do not join two distinct blocks of $F$.

Let $B_1,\dots,B_r$ be the blocks of $F$. Every entry of $F$ has its direction inside a single $B_s$ and hence lies in the subgroup $\Sa_{B_s}$, so
\[
\langle F\rangle\ \le\ \Sa_{B_1}\Sa_{B_2}\cdots\Sa_{B_r}\ \le\ \Sa_N.
\]
Now let $s_{e_i-e_j,k}\in\langle F\rangle$ be arbitrary. Then $\lin(s_{e_i-e_j,k}) = (i\ j) \in S_{B_1}\times\dots\times S_{B_r}$ and hence maps every $B_s$ to itself. But a transposition $(i\ j)$ preserves a set $B$ setwise if and only if either $i,j\in B$ or $i,j\notin B$. Since the $B_s$ partition $\{1,\dots,N\}$, the indices $i$ and $j$ must lie in the same block. So every edge of $G(\langle F\rangle)$ lies inside a block of $F$.

Consequently $G(\langle F\rangle)$ has no edge joining two distinct blocks, while each block is already connected in the subgraph $G(F)$, which shows that $G(F)$ and $G(\langle F\rangle)$ have the same connected components.
\end{proof}

By Lemma~\ref{lem:blocks} we can restrict ourselves to the case of a single block.

\begin{definition}\label{conv:irred}
We say that $w\in\Sa_N$ is \emph{irreducible} if $q(w)=1$, that is, if the only balanced partition is the one with a single part. 
\end{definition}

By Lemma~\ref{lem:blocks} every reduced reflection factorization of an irreducible element has a connected direction graph. In Sections~\ref{ch:length} and \ref{ch:main} we first treat irreducible elements and treat the general case afterwards.

Note that by Lemma~\ref{lem:nonzero}, if $w$ is irreducible and $c\ge2$ then all winding numbers
are nonzero. If $c=1$ then $\omega_1=0$ by \eqref{eq:sumzero} and $w$ is of finite order by
Lemma~\ref{lem:winding}(ii).

\begin{remark}\label{rem:graphs}
We defined the direction graph in terms of the directions of reflections. Nevertheless, since the linear part of a reflection with direction $e_i-e_j$ is given by the transposition $(i\ j)$, we can associate a graph in the same way to a set of transpositions, that is, for each transposition $(i\ j)$ there is an edge between $i$ and $j$. We will do so in Sections \ref{ch:length} and \ref{ch:main}.
\end{remark}

\section{Reflection subgroups of the affine symmetric group}\label{ch:sub}

Theorem~B asserts that the subgroup generated by a reduced reflection factorization determines its Hurwitz orbit. We therefore investigate reflection subgroups in type $\widetilde A_n$.

\begin{definition}
A subgroup $W'\le\Sa_N$ is a \emph{reflection subgroup} if it is generated by the reflections it contains, i.e.\ $W'=\langle T\cap W'\rangle$. For $\gamma\in\Phi$ we call
\[
K_\gamma(W')=\{k\in\ZZ\ :\ s_{\gamma,k}\in W'\}
\]
the \emph{level set} of $W'$ in direction $\gamma$.
\end{definition}

A general theory of reflection subgroups is due to Dyer \cite{Dyer90}. For our purposes the explicit description obtained below is more convenient.

\subsection{Subgroups with full linear part}\label{sec:structure}

\begin{lemma}\label{lem:levelcoset}
Let $W'\le\Sa_N$ be a subgroup and $\gamma\in\Phi$ with $K_\gamma(W')\ne\emptyset$. Then
$K_\gamma(W')$ is a coset $c_\gamma+e_\gamma\ZZ$ for some $c_\gamma\in\ZZ$ and $e_\gamma\ge0$.
\end{lemma}

\begin{proof}
Write $K=K_\gamma(W')$.

\emph{Step 1.}
Let $k,k'\in K$, so that $s_{\gamma,k},s_{\gamma,k'}\in W'$. By Lemma~\ref{lem:conj}(ii) we obtain
$$
s_{\gamma,k}\,s_{\gamma,k'}\,s_{\gamma,k} = s_{\gamma, 2k-k'} \in W' ~\Rightarrow ~ 2k-k' \in K.
$$ 

\emph{Step 2.} Choose any $c\in K \ne\emptyset$, and put $A=\{k-c : k \in K\}$, so that $0\in A$. If $a,b\in A$ then $c+a,c+b\in K$, hence by Step~1 we have 
$$
2(c+a)-(c+b)=c+(2a-b)\in K ~\Rightarrow ~ 2a-b\in A.
$$
It therefore suffices to show that $A=d\ZZ$ for some $d\ge0$ as this would show $K=c+d\ZZ$, with $c_\gamma=c$ and $e_\gamma=d$.

\emph{Step 3.} For $a\in A$ let
$r_a:\ZZ\to\ZZ$, $r_a(x)=2a-x$: by Step~2 $r_a$ is a map from $A$ to $A$. Composing two such maps gives a translation:
\[
r_a\big(r_b(x)\big)=2a-(2b-x)=x+2(a-b).
\]
Hence $A+2(a-b)\subseteq A$ for all $a,b\in A$. Interchanging $a$ and $b$ we obtain $A + 2(b-a) \subseteq A$ which is equivalent to $A \subseteq A +2(a-b)$. Hence we have $A = A +2(a-b)$, that is, $A$ is invariant under translation by every element of the subgroup $\langle 2(a-b) : a,b \in A \rangle \leq \ZZ$.

\emph{Step 4.} Let $d\ge0$ be the non-negative generator of the subgroup $\langle A\rangle\le\ZZ$. Since $0\in A$, every $a\in A$ is itself a difference $a-0$ of elements of $A$, so $\langle a-b: a,b\in A\rangle=\langle A\rangle=d\ZZ$. Hence we can write $d = a-b$ for $a,b \in A$. Since $0 \in A$, we have $0 + 2(a-b) = 2d \in A$ by Step~3. Hence we have $2d\ZZ \subseteq A$. Since $A\subseteq\langle A\rangle$ we obtain
\[
2d\ZZ\ \subseteq\ A\ \subseteq\ d\ZZ.
\]

\emph{Step 5.} Suppose first $d>0$. The quotient $d\ZZ/2d\ZZ$ has exactly two elements, so a subset of $d\ZZ$ containing $2d\ZZ$ is either $2d\ZZ$ or all of $d\ZZ$. The first case is impossible since it would give $\langle A\rangle=2d\ZZ$, whereas $\langle A\rangle=d\ZZ$ by the choice of $d$, and $2d\ZZ\ne d\ZZ$ for $d>0$. Hence $A=d\ZZ$. If $d=0$ then $\langle A\rangle=\{0\}$ forces $A=\{0\}=d\ZZ$ as well. In both cases $A=d\ZZ$, which completes the proof.
\end{proof}

\begin{remark}\label{rem:e0}
The case $e_\gamma=0$ occurs precisely when $W'$ contains exactly one reflection with
direction $\gamma$.
\end{remark}

The following theorem will be important in Section~\ref{ch:main}.

\begin{theorem}\label{thm:structure}
Let $W'\le\Sa_N$ be a reflection subgroup with $\lin(W')=S_N$. Then there are an integer $e\ge0$
and a vector $\nu\in\ZZ^N$ such that
\[
T\cap W'=\{\,s_{\gamma,k}\ :\ \gamma\in\Phi,\ k\equiv(\nu|\gamma)\ (\mathrm{mod}\ e)\,\}
\qquad\text{and}\qquad
W'=t_\nu\,\Gamma_e\,t_\nu^{-1},
\]
where
\[
\Gamma_e:=\{\,t_\mu u\ :\ \mu\in eQ,\ u\in S_N\,\}\ \le\ \Sa_N .
\]
Moreover
\[
\{\,\mu\in Q\ :\ t_\mu\in W'\,\}=eQ ,
\]
and both $e$ and the homomorphism $\gamma\mapsto(\nu|\gamma)\bmod e$ are uniquely determined
by $W'$.
\end{theorem}

Before proving the theorem, we want to add some comments.

\begin{remark} \label{rem:BeforeProof44}
$\phantom{0}$
\begin{enumerate}[label=\textup{(\roman*)}]
\item $\Gamma_e$ is indeed a subgroup of $\Sa_N$. To see this, note that $(t_\mu u)(t_{\mu'}u')=t_{\mu+u(\mu')}(uu')$. Hence it suffices that $u(eQ)=eQ$ for every $u\in S_N$,
which holds because $S_N$ permutes the coordinates and hence maps $Q$ to $Q$. For $e=1$ we get
$\Gamma_1=\Sa_N$, and for $e=0$ we get $\Gamma_0=S_N$.

\item The conjugating element $t_\nu$ need not lie in $\Sa_N$, but in the larger
group $\Sa_N^{+}=S_N\ltimes\ZZ^N$. See the discussion after Lemma~\ref{lem:winding}.

\item The last equality can be reformulated as $W'\cap\mathcal T=\iota(eQ)$.
\end{enumerate}
\end{remark}

\begin{proof}[Proof of Theorem \ref{thm:structure}]
$\phantom{0}$

\emph{Step 1.} Let
\[
D=\{\gamma\in\Phi\ :\ K_\gamma(W')\ne\emptyset\}
\]
be the set of directions of the reflections contained in $W'$. We first show $D=\Phi$. Since $\lin(W')=S_N\ne1$ the group $W'$ is nontrivial, and being a reflection subgroup it contains a reflection, so $D\ne\emptyset$. 

If $s_{\alpha,k},s_{\gamma,l}\in W'$ then also 
\begin{equation}\label{eq:levelconj2}
s_{\alpha,k}s_{\gamma,l}s_{\alpha,k} = s_{s_{\alpha}(\gamma), l-(\gamma \mid \alpha)k}\in W'
\end{equation}
by Lemma~\ref{lem:conj}(ii). Hence $s_{\alpha}(D) \subseteq D$. Applying $s_{\alpha}$ to this inclusion yields the other inclusion, thus $s_{\alpha}(D) = D$.

Finally, $W'$ is generated by the reflections it contains, so $\lin(W')$ is generated by their linear parts, that is, $S_N = \lin(W')=\langle s_\alpha:\alpha\in D\rangle$. Thus $D$ is a nonempty subset of $\Phi$ stable under all reflections $s_{\alpha}$ with $\alpha \in D$. But since these reflections generate $S_N$, $D$ is stable under $S_N$. Since $S_N$ acts transitively, we conclude $D=\Phi$.

By Lemma~\ref{lem:levelcoset} we may therefore write $K_\gamma(W')=c_\gamma+e_\gamma\ZZ$ for every $\gamma\in\Phi$. Let $s_{\alpha,k}\in W'$. Conjugating the reflections with
direction $\gamma$ by $s_{\alpha,k}$ yields by \eqref{eq:levelconj2} that $l -(\gamma \mid \alpha) k \in K_{s_{\alpha}(\gamma)}(W')$. Since $l \in K_{\gamma}(W')$ we obtain
\begin{equation}\label{eq:levelconj}
K_{s_\alpha(\gamma)}(W')\supseteq K_\gamma(W')-(\gamma|\alpha)\,k,
\end{equation}
Likewise, we have
\begin{align*}
    s_{\alpha,k}s_{s_{\alpha}(\gamma),l}s_{\alpha,k} = s_{\gamma, l+(\alpha \mid \gamma)k} ~\Rightarrow ~ l+(\alpha \mid \gamma)k \in K_{\gamma}(W').
\end{align*}
This gives the reverse inclusion of \eqref{eq:levelconj} and so \eqref{eq:levelconj} is an equality. We therefore have
\begin{equation}\label{eq:levelconj3}
c_{s_{\alpha}(\gamma)} + e_{s_{\alpha}(\gamma)}\ZZ = K_{s_{\alpha}(\gamma)}(W')= K_{\gamma}(W') - (\gamma \mid \alpha) k = c_{\gamma} + e_{\gamma} \ZZ - (\gamma \mid \alpha) k.
\end{equation}
Comparing these cosets yields $e_{s_\alpha(\gamma)}=e_\gamma$. Since all roots lie in one $S_N$-orbit, as just noted, and since $\lin(W')=S_N$, the number $e:=e_\gamma$ does not depend on $\gamma$.

Both cases $e=0$ and $e\ge1$ genuinely occur. By Lemma~\ref{lem:levelcoset} and Remark~\ref{rem:e0}, $e=0$ means that $W'$ contains exactly one
reflection per direction (which happens for instance for $W'=S_N$). Conversely,
\begin{equation}\label{eq:egeq1}
e\ge1\iff W'\ \text{contains two reflections with the same direction and different levels.}
\end{equation}

\emph{Step 2.} Since $s_{\alpha, k} \in W'$, by Lemma \ref{lem:levelcoset} we have 
\[
k \in K_{\alpha}(W') = c_{\alpha} + e_{\alpha} \ZZ = c_{\alpha} + e \ZZ. 
\]
Hence $k \equiv c_\alpha \pmod e$. Considering \eqref{eq:levelconj3} modulo $e$ therefore gives for all $\alpha$ and $\gamma$
\begin{equation}\label{eq:offsets}
c_{s_\alpha(\gamma)} \equiv c_\gamma-(\gamma \mid \alpha)\,c_\alpha \pmod e .
\end{equation}

Fix the simple roots $\alpha_i=e_i-e_{i+1}$, $i=1,\dots,N-1$. They form a $\ZZ$-basis of $Q$, so
there is a unique group homomorphism
\[
\psi:Q\longrightarrow\ZZ/e\ZZ,\qquad \psi(\alpha_i)=c_{\alpha_i}\bmod e .
\]
We claim $\psi(\gamma)\equiv c_\gamma$ for every root $\gamma$. Every root can be written as $\gamma=s_{\alpha_{i_1}}\cdots s_{\alpha_{i_k}}(\alpha_j)$, and we induct on $k$. For $k=0$ this is the definition of $\psi$. For the induction step write $\gamma=s_\alpha(\gamma')$ with $\alpha$ a simple root and $\gamma'$ satisfying the claim. Then we have
\[
c_\gamma=c_{s_\alpha(\gamma')} \stackrel{\eqref{eq:offsets}}{\equiv} c_{\gamma'}-(\gamma'|\alpha)c_\alpha \stackrel{\text{(induction)}}{\equiv} \psi(\gamma')-(\gamma'|\alpha)\psi(\alpha)=\psi\big(\gamma'-(\gamma'|\alpha)\alpha\big)
=\psi(s_{\alpha}(\gamma')) = \psi(\gamma)
\]

\emph{Step 3.} We construct $\nu\in\ZZ^N$ with $(\nu \mid  \gamma)\equiv\psi(\gamma)\pmod e$ for all $\gamma\in\Phi$. Choose integers $b_1,\dots,b_{N-1}$ representing the classes $\psi(\alpha_i)\in\ZZ/e\ZZ$ and define $\nu\in\ZZ^N$ by
\[
\nu_N=0,\qquad \nu_i=b_i+b_{i+1}+\dots+b_{N-1}\quad(1\le i\le N-1),
\]
so that $\nu_i-\nu_{i+1}=b_i$ for all $i$. Then $(\nu \mid \alpha_i)=\nu_i-\nu_{i+1}=b_i\equiv\psi(\alpha_i)$, and since both $\gamma\mapsto(\nu|\gamma)\bmod e$ and $\psi$ are homomorphisms on $Q$ agreeing on the basis $\alpha_1,\dots,\alpha_{N-1}$, they agree on all of $Q$, in particular on all roots. 

Combining with Step~1 and Step~2, we obtain
\begin{align*}
T' := T\cap W' = \{ s_{\gamma,k} : \gamma \in \Phi, ~k \in K_{\gamma}(W')\} & = \{ s_{\gamma,k} : \gamma \in \Phi, ~k \equiv c_{\gamma} ~\text{mod } e\}\\
& = \{ s_{\gamma,k} : \gamma \in \Phi, ~k \equiv \psi(\gamma) ~\text{mod } e\}\\
& = \{s_{\gamma,k}\ :\ k\equiv(\nu \mid \gamma) ~\text{mod } e\}.
\end{align*}

\emph{Step 4.} By Lemma~\ref{lem:conj}(i), conjugation
by $t_\nu^{-1}=t_{-\nu}$ preserves the direction of a reflection and lowers its level by $(\nu \mid \gamma)$. Applying this to $T'$ yields
\[
t_\nu^{-1}\,T'\,t_\nu
=\big\{\,s_{\gamma,\;k-(\nu|\gamma)}\ :\ k\equiv(\nu \mid \gamma) ~\text{mod } e\,\big\}
=\{\,s_{\gamma,l}\ :\ \gamma\in\Phi,\ e\mid l\,\}.
\]
Note that the translation vector $\nu$ lies in $\ZZ^N$ (see part (ii) of Remark \ref{rem:BeforeProof44}).

Since conjugation by $t_\nu^{-1}$ is an automorphism of $\Sa_N$, it commutes with forming
generated subgroups, i.e.\ $t_\nu^{-1}\langle T'\rangle t_\nu=\langle t_\nu^{-1}T't_\nu\rangle$
and $W'=\langle T\cap W'\rangle=\langle T'\rangle$ because $W'$ is a reflection subgroup. Hence we have 
\[
t_\nu^{-1} W' t_\nu = \langle t_\nu^{-1} T' t_\nu \rangle = \big\langle\, s_{\gamma,l}\ :\ \gamma\in\Phi,\ e\mid l \,\big\rangle
\]
It therefore suffices to prove
\[
\big\langle\, s_{\gamma,l}\ :\ \gamma\in\Phi,\ e\mid l \,\big\rangle=\Gamma_e ,
\]
for then $t_\nu^{-1}W't_\nu=\Gamma_e$, that is, $W'=t_\nu\Gamma_e t_\nu^{-1}$.

``$\subseteq$'': $s_{\gamma,ek'}=t_{ek'\gamma}s_\gamma$ with $ek'\gamma\in eQ$, so each generator lies in $\Gamma_e$. 

``$\supseteq$'': the left-hand side contains every $s_\gamma=s_{\gamma,0}$,
hence contains $S_N$, and by Lemma~\ref{lem:pairs} it contains
$s_{\gamma,e}s_{\gamma,0}=t_{e\gamma}$ for every root $\gamma$. As the roots generate $Q$, the
translations $t_{e\gamma}$ generate $\iota(eQ)$. Products of these give all of $\Gamma_e$.

\emph{Step 5.} We have
\[
W'\cap\mathcal T=t_\nu\big(\Gamma_e\cap\mathcal T\big)t_\nu^{-1}=\Gamma_e\cap\mathcal T=\iota(eQ),
\]
where the first equality is Step~4 and the second equality is due to the fact that each element in $\Gamma_e\cap\mathcal T$ is a translation and therefore commutes with $t_\nu$. By using that $\iota$ is injective, we obtain the asserted equality $\{\mu\in Q: t_\mu\in W'\}=eQ$. In particular $e$ is determined by $W'$, and so is the homomorphism $\gamma\mapsto(\nu|\gamma)\bmod e$, by the description of $T\cap W'$ in Step~3.
\end{proof}

\begin{remark}\label{rem:glide}
An element of $\Sa_N$ whose linear part is a reflection need not itself be a reflection: $t_\mu s_\gamma$ is a reflection precisely when $\mu\in\RR\gamma$, and otherwise it is a glide reflection of infinite order. For example, consider the element $t_{e_1-e_2}s_{e_3-e_4,0}$, whose linear part is the transposition $(3\,4)$ but whose reflection length is $3$.
\end{remark}

\subsection{About the root lattice in type $A$}\label{sec:tree}

\begin{lemma}\label{lem:tree}
Let $R$ be a set of roots in the root system $\Phi$ of type $A_{N-1}$ whose $\RR$-span has dimension $N-1$.
Then $\ZZ R=Q$ and $\langle s_\gamma:\gamma\in R\rangle=S_N$.
\end{lemma}

\begin{proof}
The reflection subgroup $W_R := \langle s_\gamma:\gamma\in R\rangle$ is of rank $N-1$. But in type $A_{N_1}$ the only reflection subgroup of rank $N-1$ is $W_{\Phi}$ itself. Hence we have $W_R = S_N$. By \cite[Proposition~5.10]{BGRW17} this also implies $\ZZ R=Q$.
\end{proof}

\begin{lemma}[Levels and translations]\label{lem:levels}
Let $\gamma_1,\dots,\gamma_m\in\Phi$ be linearly independent roots, let $k_1,\dots,k_m\in\ZZ$, and
put
\[
v_j=\big(s_{\gamma_1}s_{\gamma_2}\cdots s_{\gamma_{j-1}}\big)(\gamma_j)\qquad(1\le j\le m),
\]
so that $v_1=\gamma_1$. Then:
\begin{enumerate}[label=\textup{(\roman*)}]
\item $\displaystyle s_{\gamma_1,k_1}s_{\gamma_2,k_2}\cdots s_{\gamma_m,k_m}
=t_{\,\sum_{j=1}^m k_jv_j}\ s_{\gamma_1}s_{\gamma_2}\cdots s_{\gamma_m}$;
\item $v_j\in\gamma_j+\ZZ\gamma_1+\dots+\ZZ\gamma_{j-1}$ for $j>1$;
\item  $v_1,\dots,v_m$ is a $\ZZ$-basis of the lattice $M:=\ZZ\gamma_1+\dots+\ZZ\gamma_m$;
\item the map $\ZZ^m\to M$, $k\mapsto\sum_j k_jv_j$, is a bijection.
\end{enumerate}
\end{lemma}

\begin{proof}
$\phantom{0}$

(i) This is \cite[Lemma~2.11]{Weg20}.

(ii) Since $(\gamma_i \mid \gamma_k) \in \ZZ$, this is a direct consequence of 
$$
s_{\gamma_k}(\gamma_i) = \gamma_i -(\gamma_i \mid \gamma_k) \gamma_k \in \gamma_i + \ZZ \gamma_k. 
$$

(iii) is a direct consequence of (ii), while (iv) is a direct consequence of (iii).
\end{proof}

\begin{remark}\label{rem:D4}
Lemma~\ref{lem:tree} is special to type $A$, and does in general not generalise to other types. The reason for this is that, in general, for a root system $\Phi$ there exists a proper subsystem $\Psi$ of $\Phi$ of the same rank as $\Phi$. Consider the root system of type $D_4$, that is,
\[
\Phi(D_4)=\{\pm e_i\pm e_j\ :\ 1\le i<j\le4\},\qquad Q(D_4)=\Big\{x\in\ZZ^4:\sum_i x_i\ \text{even}\Big\},
\]
and take the four roots $R=\{e_1-e_2,\ e_1+e_2,\ e_3-e_4,\ e_3+e_4\}$, which are pairwise orthogonal and span $\RR^4$. Then $\ZZ R$ is a proper sublattice of $Q(D_4)$ of index $2$ (for instance $e_1+e_3\in Q(D_4)$ but
$e_1+e_3\notin\ZZ R$), and the group generated by the four corresponding reflections is a proper subgroup of the Weyl group $W(D_4)$.
\end{remark}

\section{Reflection length in the affine symmetric group}\label{ch:length}

In this section we prove Theorem~A. The lower bound comes from a classical theorem about
factorizations of permutations, which we now recall.

\begin{theorem}[{Hurwitz \cite{Hur91}}]\label{thm:genus}
Let $M\ge1$, let $u\in S_M$ have $c(u)$ cycles, and let $\tau_1,\dots,\tau_k$ be transpositions
with
\[
\tau_1\tau_2\cdots\tau_k=u\qquad\text{and}\qquad\langle\tau_1,\dots,\tau_k\rangle = S_M.
\]
Then $k\ge M+c(u)-2$, and this bound is attained.
\end{theorem}

Proofs of the previous theorem can also be found in \cite[\S5.2]{LZ04} or \cite{GJ97}.

\begin{theorem}\label{thm:length}
Let $w\in\Sa_N$, let $c$ be the number of cycles of $\lin(w)$ and $q=q(w)$ the maximal number of
parts of a balanced partition. Then
\[
\ell_T(w)=N+c-2q .
\]
\end{theorem}

\begin{proof}
We show that $N+c-2q$ is a lower and an upper bound for $\ell_T(w)$, where $w = t_{\mu}u$.

\emph{Lower bound.} Let $F\in\Red_T(w)$ with blocks $B_1,\dots,B_r$, and let $c_s$ be the number of cycles of $u$ contained in $B_s$. By Lemma~\ref{lem:blocks}(i) these numbers are well defined and $\sum_s|B_s|=N$, $\sum_sc_s=c$. Fix $s$ and apply $\lin$ to the subtuple $F_s$. We obtain a factorization of $\lin(w_s)\in S_{B_s}$ into $\ell_T(w_s)$ transpositions, all of whose directions are edges inside $B_s$. By construction $B_s$ is a connected component of the direction graph, so the corresponding transpositions generate the subgroup $S_{B_s}$ (see \cite[Lemma 3.10.1]{GR01}). Also
$\lin(w_s)$ has exactly $c_s$ cycles. Theorem~\ref{thm:genus} therefore gives
\[
\ell_T(w_s)\ \ge\ |B_s|+c_s-2 .
\]
Summing over $s$ gives
\[
\ell_T(w)\stackrel{\text{Lemma~\ref{lem:blocks}(ii)}}{=}\sum_{s=1}^r\ell_T(w_s)\ \ge\ \sum_{s=1}^r\big(|B_s|+c_s-2\big)=N+c-2r\ \stackrel{\text{Lemma~\ref{lem:blocks}(iii)}}{\ge}\ N+c-2q.
\]

\emph{Upper bound.} We construct a factorization of length $N+c-2q$. Fix a maximal balanced partition $\pi\in\mathcal P(w)$. It suffices to treat only one part $B\in\pi$ and to concatenate, since the factors belonging to different parts commute and the lengths add as $\sum_{B\in\pi}(|B|+\#B-2)=N+c-2q$.

So let $B\in\pi$ consist (after a possible renumbering) of the cycles $C_1,\dots,C_b$, and choose a representative $r_t\in C_t$ for each $t$. Let $\mu^{B}$ denote the restriction of $\mu$ to the index set $B$, i.e.\ the vector agreeing with $\mu$ on $B$ and
vanishing elsewhere. Since $\pi$ is balanced,
\[
\sum_{i\in B}\mu^{B}_i=\sum_{i\in B}\mu_i=\sum_{C_t\subseteq B}\omega_t=0 ,
\]
so $\mu^{B}\in Q(B)$ and $w_B:=t_{\mu^{B}}\,u|_{B}$ lies in $\Sa_B$. As $\mu=\sum_{B\in\pi}\mu^{B}$
and the factors act on disjoint sets of coordinates, $w=\prod_{B\in\pi}w_B$. We now factor $w_B$. Put
\[
\gamma_t=e_{r_1}-e_{r_t}\quad(2\le t\le b),\qquad d_t=-\omega_t,\qquad
\nu=\sum_{t=2}^{b}d_t\gamma_t .
\]
Set 
\begin{equation} \label{eq:DefRho}
\rho=\mu^{B}-u(\nu).
\end{equation}
For $2\le t\le b$ the vector $\nu$ has coordinate sum $-d_t$ on $C_t$ (only the term $d_t\gamma_t$ contributes, through the entry $-d_t$ at position $r_t$) and coordinate sum $\sum_{t\ge2}d_t$ on $C_1$. Since $u$ preserves each cycle, $u(\nu)$ has the same cycle sum as $\nu$, i.e. $\sum_{i \in C_t} (u(\nu))_i = \sum_{i \in C_t} \nu_i$ for every $t$. Hence for $2\le t\le b$
\[
\sum_{i\in C_t}\rho_i = \sum_{i\in C_t} (\mu^B)_i - \sum_{i\in C_t} (u(\nu))_i = \sum_{i\in C_t} \mu_i - \sum_{i\in C_t} \nu_i  = \omega_t -(-d_t)= \omega_t -\omega_t = 0 
\]
and for $t=1$, using $\sum_{t}\omega_t=0$,
\[
\sum_{i\in C_1}\rho_i=\omega_1-\sum_{t\ge2}d_t=\omega_1+\sum_{t\ge2}\omega_t=0 .
\]
So the restriction $\rho_t$ of $\rho$ to $C_t$ lies in $Q(C_t)$ for every $t$.

Now fix $t$ and consider $A_t:=t_{\rho_t}\,u|_{C_t}\in\Sa_{C_t}$. Choose $|C_t|-1$ transpositions inside $C_t$ whose direction graph is a tree and whose product is the cycle $u|_{C_t}$. Such a choice exists by Theorem~\ref{thm:genus} (with $M=|C_t|$, $c(u)=1$, so $k=|C_t|-1$). The roots corresponding to these transpositions are $|C_t|-1$ linearly independent roots spanning $Q(C_t)\otimes\RR$, hence generate $Q(C_t)$ by Lemma~\ref{lem:tree} applied inside $C_t$. We may therefore choose levels $k_1,\dots,k_{|C_t|-1}$ so that the corresponding reflections multiply to $A_t$. Indeed, by Lemma~\ref{lem:levels}(i) the product of the reflections with levels $k_j$ equals $t_{\sum_jk_jv_j}\,u|_{C_t}$, where $v_j$ is the root as defined in Lemma~\ref{lem:levels}. By part (iii) of that same Lemma the $v_j$ form a $\ZZ$-basis of the lattice generated by the roots, which is $Q(C_t)$, so by Lemma~\ref{lem:levels}(iv) there is a (unique) choice of $k \in \ZZ^{C_t}$ with $\sum_jk_jv_j=\rho_t$, giving the product $t_{\rho_t}u|_{C_t}=A_t$.

Finally, by Lemma~\ref{lem:pairs} the pair $\big(s_{\gamma_t,d_t},s_{\gamma_t,0}\big)$ has
product $t_{d_t\gamma_t}$. Concatenating everything, the tuple
\[
\Big(\underbrace{\text{$|C_1|-1$ reflections}}_{\text{product }A_1},\dots,
\underbrace{\text{$|C_b|-1$ reflections}}_{\text{product }A_b},\
\underbrace{\big(s_{\gamma_2,d_2},s_{\gamma_2,0}\big),\dots,\big(s_{\gamma_b,d_b},s_{\gamma_b,0}\big)}_{\text{product }t_{\sum_{t=2}^b d_t\gamma_t} = t_{\nu}}\Big)
\]
has length $\sum_t(|C_t|-1)+2(b-1)=|B|+b-2$ and product
\[
A_1\cdots A_b\cdot t_\nu=t_{\sum_t\rho_t}\,u|_B\ t_\nu=t_{\rho+u(\nu)}\,u|_B\stackrel{\eqref{eq:DefRho}}{=}t_{\mu^B}u|_B=w_B ,
\]
where we used that the $A_t$ act on disjoint sets of coordinates and the multiplication rule $(t_\rho v)(t_\nu)=t_{\rho+v(\nu)}v$.
\end{proof}

\begin{example}
Let $N=4$ and $w=t_\lambda$ with $\lambda=(2,-1,-1,0)$. By Example~\ref{ex:balanced}(1) we have $c=4$ and $q=2$, so $\ell_T(w)=4+4-4=4$.
\end{example}

\begin{remark}\label{rem:lmps}
Note that the formular for the reflection length of Theorem~\ref{thm:length} is precisely the formula of \cite[Theorem~4.25]{LMPS19}. Lewis, McCammond, Petersen and Schwer derive this result from a more general result computing the reflection length in an arbitrary affine Coxeter group \cite[Theorem~A]{LMPS19}. 

We nevertheless provide a proof of Theorem~\ref{thm:length}. First, since it is independent. Second, and more importantly for our purposes, the proof of the upper bound constructs an explicit reduced factorization which we will need in Section~\ref{ch:main}.
\end{remark}

\begin{corollary}\label{cor:elliptic}
If $w$ is of finite order, i.e.\ all winding numbers vanish, then $\ell_T(w)=N-c$, which is the reflection length of $\lin(w)$ in the finite symmetric group $S_N$.
\end{corollary}

\section{The normal form and the main theorem}\label{ch:main}

By Lemma~\ref{lem:blocks} we assume throughout Sections~\ref{sec:normalform} and \ref{sec:moves} that $w \in \Sa_N$ is irreducible in the sense of Definition~\ref{conv:irred}, i.e.\ $q(w)=1$. The general case will be considered in Section~\ref{sec:count}. Thus by Theorem~\ref{thm:length} we assume
\[
m:=\ell_T(w)=N+c-2 .
\]

\subsection{The star--connector normal form}\label{sec:normalform}

The plan is to use the theorem of Clebsch and Hurwitz to normalise the \emph{permutation}
data of a factorization, and then to see how much freedom is left in the \emph{levels}.

\begin{theorem}[{Hurwitz \cite{Hur91}, Kluitmann \cite[Theorem 1]{Klu88}}]\label{thm:CH}
Let $u\in S_M$ with $c(u)$ cycles. Among all factorizations of $u$ into transpositions that generate $S_M$, the minimal number of factors is $k=M+c(u)-2$, and the set of factorizations attaining this minimum forms a single Hurwitz orbit.
\end{theorem}

Again see also \cite[\S5.2]{LZ04}. Note that $k$ is exactly the minimal length permitted by
Theorem~\ref{thm:genus}.

\begin{remark}\label{rem:CHhyp}
The hypothesis that the factors generate $S_M$ cannot be dropped. The factorizations considered in Theorem~\ref{thm:CH} have length $M+c(u)-2$, which exceeds the reflection length $\ell_T(u)=M-c(u)$ of $u$ in $S_M$ by $2(c(u)-1)$. They are reduced only when $u$ is a single cycle.
\end{remark}

We now fix a \emph{normal form} for reflection factorizations in $\Sa_N$ which we have already seen in the proof of Theorem~\ref{thm:length}.

\begin{definition}\label{def:normalform}
Choose a representative $r_t\in C_t$ for each cycle of $u=\lin(w)$ and put
$\gamma_t=e_{r_1}-e_{r_t}$ for $2\le t\le c$. For each $t$ fix a tuple $\mathcal T_t$ of
$|C_t|-1$ transpositions supported on $C_t$, forming a tree and with product the cycle
$u|_{C_t}$.

A tuple $F\in T^m$ is in \emph{star--connector form} if
\[
\lin(F)=\Big(\mathcal T_1,\ \mathcal T_2,\ \dots,\ \mathcal T_c,\
\big(\tau_{2},\tau_{2}\big),\ \big(\tau_3,\tau_3\big),\ \dots,\ \big(\tau_{c},\tau_{c}\big)\Big),
\qquad \tau_t=(r_1\ r_t).
\]
In words: the first $N-c$ entries factor the cycles one after the other, and the last $2(c-1)$
entries consist of $c-1$ pairs of transpositions, which we call \emph{connector pairs}. In particular the pair of reflections in $F$ corresponding to the connector pair $(\tau_t, \tau_t)$ has a common direction $\gamma_t$. 
More precisely, the pair of reflections in $F$ lifting the connector pair $(\tau_t, \tau_t)$ is of the form $(s_{\gamma_t, x_t}, s_{\gamma_t, y_t})$. We call $y_t$ the \emph{position} and $x_t-y_t$ the \emph{gap} of the connector pair. We call $y = (y_2, \dots, y_c)$ the \emph{position vector}. By abuse of notation, we call both $(\tau_t, \tau_t)$ and $(s_{\gamma_t, x_t}, s_{\gamma_t, y_t})$ connector pair, whereby the context makes it clear which pair is meant. An entry of a connector pair is called \emph{connector}.
\end{definition}

The transpositions $\tau_t$ ($2 \leq t \leq c$) from the connector pairs form a star with centre $r_1$, hence the name. 
Also recall the defintion of $\ZZ^{\{2,\dots,c\}}$ as introduced before the proof of Lemma~\ref{lem:winding}. We choose to index coordinates of $y$ by the labels $2,\dots,c$ rather than by $1,\dots,c-1$, so that the coordinate $y_t$, the direction $\gamma_t$, the cycle $C_t$ and the winding number $\omega_t$ all carry the same index.

\begin{definition}\label{def:Fy}
Assume $q(w)=1$, write $w=t_\mu u$ and keep the choices made in Definition~\ref{def:normalform}. Put
\[
d_t:=-\omega_t\quad(2\le t\le c),\qquad \nu:=\sum_{t=2}^{c}d_t\gamma_t,\qquad \rho:=\mu-u(\nu),
\]
and let $\rho_t$ be the restriction of $\rho$ to the index set $C_t$. Let
$\gamma^t_1,\dots,\gamma^t_{m_t}$, $m_t=|C_t|-1$, be the roots of the transpositions in
$\mathcal T_t$, in order, and set $v^t_j:=\big(s_{\gamma^t_1}\cdots s_{\gamma^t_{j-1}}\big)(\gamma^t_j)$.
Finally let $k^t=(k^t_1,\dots,k^t_{m_t})\in\ZZ^{m_t}$ be the unique vector with
$\sum_{j}k^t_jv^t_j=\rho_t$.

For $y\in\ZZ^{\{2,\dots,c\}}$ define
\[
F(y):=\Big(\,s_{\gamma^1_1,k^1_1},\dots,s_{\gamma^1_{m_1},k^1_{m_1}},\dots, s_{\gamma^c_{1},k^c_{1}}, \dots, s_{\gamma^c_{m_c},k^c_{m_c}},\ \
s_{\gamma_2,\,y_2+d_2},\ s_{\gamma_2,\,y_2},\ \dots,\ s_{\gamma_c,\,y_c+d_c},\ s_{\gamma_c,\,y_c}\,\Big),
\]
a tuple of $m=N+c-2$ reflections in star--connector form with connector gaps $d_t$ and positions
$y_t$.
\end{definition}

Let us check that $F(y)$ is well defined: By the proof of Theorem~\ref{thm:length} we have that $\rho_t\in Q(C_t)$, and the vector $k^t$ exists and is uniquely determined. The latter because the roots $\gamma^t_j$ are linearly independent and generate $Q(C_t)$ (Lemma~\ref{lem:tree}), so that the $v^t_j$ are a $\ZZ$-basis of $Q(C_t)$ by Lemma~\ref{lem:levels}(iii),(iv).

\begin{theorem}\label{thm:normalform}
Assume $q(w)=1$. Then:
\begin{enumerate}[label=\textup{(\roman*)}]
\item every $F\in\Red_T(w)$ is Hurwitz equivalent to a tuple in star--connector form;
\item $F(y)\in\Red_T(w)$ for every $y\in\ZZ^{\{2,\dots,c\}}$, and conversely every tuple in
star--connector form with product $w$ equals $F(y)$ for exactly one $y$;
\item $d_t\ne0$ for every $t$.
\end{enumerate}
In particular, every reduced reflection factorization of $w$ is Hurwitz equivalent to $F(y)$ for
some $y$, and the assignment $y\mapsto F(y)$ is injective.
\end{theorem}

\begin{proof}
(i) Let $F\in\Red_T(w)$. Applying $\lin$ entrywise gives a tuple of $m=N+c-2$ transpositions
whose product is $u$. Since the direction graph of $F$ is connected by Lemma~\ref{lem:blocks} (recall $q(w)=1$), we obtain by Lemma~\ref{lem:tree} and \cite[Lemma 3.10.1]{GR01} that $u$ generates $S_N$. The tuple displayed in Definition~\ref{def:normalform} has the same length, the same product (its product is $u|_{C_1}\cdots u|_{C_c}=u$, since each connector pair contributes $\tau_t\tau_t=1$) and its directions likewise form a connected graph, since the connectors join all cycles, so it too generates $S_N$. If we call this tuple $G$, by Theorem~\ref{thm:CH} there is a braid $\beta\in \mathcal{B}_m$ with $\beta\cdot\lin(F)=G$.

Now the group homomorphism $\lin$ is equivariant with respect to the Hurwitz action. Hence $\lin(\beta\cdot F)=\beta\cdot\lin(F)=G$, so $\beta\cdot F$ is a reduced factorization of $w$ in
star--connector form.

(ii) We first show that a tuple $F$ in star--connector form has product $w$ if and only if its connector gaps are the $d_t$ and its cycle levels are the $k_t$ of Definition \ref{def:Fy}. Since the positions $y_t$ then determine $F$, this gives both assertions of (ii) at once.

Denote by $A_t\in\Sa_{C_t}$ the product of the subtuple lifting $\mathcal T_t$ and by $d_t$ the difference of the two levels of the connector pair with direction $\gamma_t$. By Lemma~\ref{lem:pairs} the product of that pair is $t_{d_t\gamma_t}$, which does not depend on $y_t$. Since translations commute, the product of all connector pairs is $t_\nu$ with $\nu=\sum_{t\ge2}d_t\gamma_t$, and
\[
w=A_1A_2\cdots A_c\cdot t_\nu .
\]
The $A_t$ act on pairwise disjoint sets of coordinates, so $A_1\cdots A_c=t_\rho\,u$ where $\rho$ is the vector whose restriction to $C_t$ is the translation part $\rho_t\in Q(C_t)$ of $A_t$. Hence
\[
\mu=\rho+u(\nu).
\]
Now sum the coordinates over the cycle $C_t$. Since $\rho_t\in Q(C_t)$ the vector $\rho$ contributes $0$, and since $u$ preserves $C_t$ the vector $u(\nu)$ contributes the same as $\nu$. Therefore
\[
\omega_t=\sum_{i\in C_t}\mu_i=\sum_{i\in C_t}\nu_i=-d_t\quad(t\ge2),
\]
because $\nu=\sum_{s\ge2}d_s\gamma_s=\sum_{s\ge2}d_s(e_{r_1}-e_{r_s})$ has $C_t$-coordinate sum $-d_t$ for $t\ge2$.

So $d_t=-\omega_t$ as claimed, and $d_t\ne0$ by Lemma~\ref{lem:nonzero} (here $q(w)=1$, and $c\ge2$ since otherwise there are no connectors at all). Note also that the case $t=1$ is automatically consistent, by
\eqref{eq:sumzero}.

Consequently $\nu$ is determined by $w$, hence so is $\rho=\mu-u(\nu)$ and therefore each $A_t$. Finally, the levels of the entries lifting $\mathcal T_t$ are determined by $A_t$. Their directions are $|C_t|-1$ linearly independent roots forming a tree inside $C_t$, hence generate $Q(C_t)$ by Lemma~\ref{lem:tree}. By Lemma~\ref{lem:levels}(i) the product of these entries with level vector $k$ equals $t_{\sum_jk_jv_j}\,u|_{C_t}$, and by Lemma~\ref{lem:levels}(iv) the assignment $k\mapsto\sum_jk_jv_j$ is injective, so $k$ is determined by $A_t$. 

So if $F$ is in star--connector form with product $w$, then its gaps and cycle levels are exactly those of Definition~\ref{def:Fy}, i.e.\ $F=F(y)$ with $y$ its vector of positions. This $y$ is uniquely determined, since distinct position vectors give distinct tuples. Conversely, for $F=F(y)$ the gaps are $d_t$, hence the product of the connector part is $t_\nu$, and the cycle levels $k^t$ give $A_t=t_{\rho_t}u|_{C_t}$ by Lemma~\ref{lem:levels}(i), so the product is $t_\rho u\,t_\nu=t_{\rho+u(\nu)}u=t_\mu u=w$. Its
length is $m=N+c-2=\ell_T(w)$ by Theorem~\ref{thm:length}, so $F(y)\in\Red_T(w)$.

(iii) This was already shown in (ii).
\end{proof}

So, up to Hurwitz equivalence, a reduced factorization of an irreducible $w$ is described by a single vector $y\in\ZZ^{\{2,\dots,c\}}$ of integers. It remains to decide when two such vectors yield Hurwitz equivalent factorizations. This is the content of the next two sections, and the answer is: exactly when they agree modulo
\[
g:=\gcd(\omega_1,\dots,\omega_c)=\gcd(d_2,\dots,d_c),
\]
the last equality because $d_t = - \omega_t$ and $\omega_1=-\sum_{t\ge2}\omega_t$ by \eqref{eq:sumzero}.

\begin{remark}
One way to generalize Theorem~\ref{thm:normalform} to other affine types would be to generalize Theorem~\ref{thm:CH} to all finite Coxeter groups, which is a conjecture of Lewis \cite[Conjecture~1]{DL24}. Also a generalization of the star--connector form to other affine types or a simplification of the arguments given here might possibly be achieved by using \cite[Corollary~1.4]{LR16}.
\end{remark}

\subsection{Moving the connectors}\label{sec:moves}

\begin{lemma}[The lattice of shifts]\label{lem:shiftlattice}
Let $d_2,\dots,d_c$ be nonzero integers and $g=\gcd(d_2,\dots,d_c)$. Let
$L\le\ZZ^{\{2,\dots,c\}}$ be the subgroup generated by the vectors
\[
d_t\,\mathbf e_t\quad(2\le t\le c)\qquad\text{and}\qquad d_s\,\mathbf e_t+d_t\,\mathbf e_s\quad(s\ne t),
\]
where $\mathbf e_t$ is the standard basis vector of $\ZZ^{\{2,\dots,c\}}$ labelled by $t$.
Then $L=g\,\ZZ^{\{2,\dots,c\}}$.
\end{lemma}

\begin{proof}
Every generator has all its entries divisible by $g$, so $L\subseteq g\ZZ^{\{2,\dots,c\}}$. Writing $d_t=gd_t'$ replaces $L$ by $gL'$ with $L'$ the analogous lattice for the $d_t'$, so we may assume $g=1$ and must show $L=\ZZ^{\{2,\dots,c\}}$.

The quotient $\ZZ^{\{2,\dots,c\}}/L$ is a finitely generated abelian group. If it were nontrivial it would admit a surjection onto $\ZZ$ or onto $\ZZ/p$ for some prime $p$, and in either case onto $\ZZ/p$. So it suffices to show that for every prime $p$ no nonzero linear form $\varphi=\sum_t u_tx_t$ over $\ZZ/p$ vanishes on $L$. Suppose $\varphi$ does vanish, i.e.
\[
u_td_t\equiv0\ \ (2 \leq t \leq c),\qquad u_td_s+u_sd_t\equiv0\ \ \text{mod }p\quad (s\ne t) .
\]
Since $\gcd(d_2,\dots,d_c)=1$ there is an index $k$ with $d_k\not\equiv0\ \text{mod } p$. The first relation with $t=k$ gives $u_k\equiv0$. The second relation, with $s=k$ and arbitrary $t$, then gives $u_td_k+u_kd_t\equiv u_td_k\equiv0$, whence $u_t\equiv0$ for every $t$. So $\varphi=0$.
\end{proof}

\begin{lemma}\label{lem:moves}
Assume $q(w)=1$ and $c\ge2$, and let $y\in\ZZ^{\{2,\dots,c\}}$. Then
\[
F(y)\ \sim\ F(y+d_t\mathbf e_t)\quad(2\le t\le c),
\qquad
F(y)\ \sim\ F(y+d_s\mathbf e_t+d_t\mathbf e_s)\quad(s\ne t).
\]
Hence $F(y)\sim F(y')$ whenever $y\equiv y'\ \text{mod } g$, where $g=\gcd(\omega_1,\dots,\omega_c)=\gcd(d_2,\dots,d_c)$.
\end{lemma}

\begin{proof}
We proceed in five steps.

\emph{Step 1: Move 1.} The connector pair with direction $\gamma_t$ is $\big(s_{\gamma_t,\,y_t+d_t},\,s_{\gamma_t,\,y_t}\big)$. Applying a simple Hurwitz move yields
\[
\big(s_{\gamma_t,\,y_t+d_t},\,s_{\gamma_t,\,y_t}\big) \sim
\big(s_{\gamma_t,\,y_t+d_t}\ s_{\gamma_t,\,y_t}\ s_{\gamma_t,\,y_t+d_t},\ s_{\gamma_t,\,y_t+d_t}\big)
=\big(s_{\gamma_t,\,y_t+2d_t},\ s_{\gamma_t,\,y_t+d_t}\big),
\]
again a connector pair with direction $\gamma_t$ and gap $d_t$, but with second level $y_t+d_t$. All other entries stay unchanged, so the result is $F(y+d_t\mathbf e_t)$.

\emph{Step 2: swapping two adjacent connector pairs.}
Suppose the connector pairs $P_a$ and $P_b$ occur next to each other, in the order $(\dots,P_a,P_b,\dots)$. Their products are the translations $\pi_a=t_{d_a\gamma_a}$ and $\pi_b=t_{d_b\gamma_b}$ (Lemma~\ref{lem:pairs}). Applying the block move \eqref{eq:blockmove} to these two blocks gives
\[
(P_a,P_b)\ \sim\ (\pi_aP_b\pi_a^{-1},\ P_a).
\]
By Lemma~\ref{lem:conj}(i) each entry of $P_b$ keeps its direction $\gamma_b$ and has its level
raised by
\[
(d_a\gamma_a\,|\,\gamma_b)=d_a(\gamma_a|\gamma_b)=d_a ,
\]
since $\gamma_a=e_{r_1}-e_{r_a}$ and $\gamma_b=e_{r_1}-e_{r_b}$ with $r_a\ne r_b$. Thus $\pi_aP_b\pi_a^{-1}$ is again a
connector pair with direction $\gamma_b$ and gap $d_b$, now with position $y_b+d_a$, while $P_a$ is unchanged. We record:
\begin{quote}
$\mathrm{sw}(a,b)$: the two pairs are interchanged, the coordinate $y_b$ is raised by $d_a$, and
all other coordinates of $y$ are unchanged.
\end{quote}
Two observations about $\mathrm{sw}(a,b)$ will be used below. First, it is realised by a Hurwitz move, so the inverse Hurwitz move realises the inverse operation: the two pairs are interchanged back and $y_b$ is lowered by $d_a$, all other coordinates being unchanged. Second, the product $\pi_a$ of a connector pair depends only on its gap $d_a$ and not on its position $y_a$. Hence the effect just recorded is the same whatever the current positions happen to be.

\emph{Step 3: Move 2 for two adjacent pairs.}
Let $s\ne t$ and suppose that $P_s$ and $P_t$ are adjacent, in the order $(P_s,P_t)$. Apply $\mathrm{sw}(s,t)$: the pairs now stand in the order $(P_t,P_s)$ and $y_t$ has been raised by $d_s$. Apply $\mathrm{sw}(t,s)$ to the same two adjacent positions: the pairs stand again in the order $(P_s,P_t)$ and $y_s$ has been raised by $d_t$. The order is therefore the one prescribed by Definition~\ref{def:normalform}, while the effect on $y$ is
\[
y\ \longmapsto\ y+d_t\mathbf e_s+d_s\mathbf e_t .
\]

\emph{Step 4: Move 2 for arbitrary pairs.}
Let $s\ne t$ and we may assume that $P_s$ occurs before $P_t$. Let $P_{a_1},\dots,P_{a_r}$ be the connector pairs strictly between them, in order. Apply $\mathrm{sw}(s,a_1),\mathrm{sw}(s,a_2),\dots,\mathrm{sw}(s,a_r)$ in this order. Each of them moves $P_s$ one position to the right and raises the position of the pair it passes by $d_s$. Afterwards, $P_s$ stands immediately before $P_t$, the coordinates $y_{a_1},\dots,y_{a_r}$ have each been raised by $d_s$, and $y_s$ and $y_t$ are unchanged. Now apply Step~3 to the adjacent pairs $(P_s,P_t)$: this raises $y_s$ by $d_t$ and $y_t$ by $d_s$, and leaves $P_s$ and $P_t$ in their original relative order. Finally apply the inverse operations of $\mathrm{sw}(s,a_r),\dots,\mathrm{sw}(s,a_1)$, in this order. By the first observation of Step~2, each of them moves $P_s$ one position back to the left and lowers the position of the pair it passes by $d_s$. By the second observation this holds even though $y_s$ has changed in the meantime, since $d_s$ has not. Afterwards, the connector pairs stand in their original order and $y_{a_1},\dots,y_{a_r}$ have returned to their original values. The effect of the whole sequence is
\[
y\ \longmapsto\ y+d_t\mathbf e_s+d_s\mathbf e_t ,
\]
which is the second assertion of the lemma.

\emph{Step 5: conclusion.}
Put
\[
S=\big\{\,v\in\ZZ^{\{2,\dots,c\}}\ :\ F(z)\sim F(z+v)\ \text{ for every }z\in\ZZ^{\{2,\dots,c\}}\,\big\}.
\]
Steps~1, 3 and~4 produce their shifts starting from an arbitrary position vector, so
\[
d_t\mathbf e_t\in S\quad(2\le t\le c),\qquad d_s\mathbf e_t+d_t\mathbf e_s\in S\quad(s\ne t).
\]
Moreover it is easy to see that $S$ is a subgroup of $\ZZ^{\{2,\dots,c\}}$. Hence $S$ contains the subgroup $L$ generated by the vectors $d_t \mathbf e_t$ $(2\le t\le c)$ and $d_s\mathbf e_t+d_t\mathbf e_sv$ $(s\ne t)$, and $L=g\,\ZZ^{\{2,\dots,c\}}$ by Lemma~\ref{lem:shiftlattice}.

Therefore, if $y\equiv y'~\text{mod } g$, then $y'-y\in g\,\ZZ^{\{2,\dots,c\}}=L\subseteq S$, and so
$F(y)\sim F(y')$.
\end{proof}

\begin{remark}
Interchanging $P_a$ and $P_b$ necessarily shifts $y_b$ by $d_a$. It is only after Lemma~\ref{lem:shiftlattice} that one sees that such a reordering does not leave the Hurwitz orbit, since $g\mid d_s$.
\end{remark}

\subsection{The main theorem}\label{sec:count}

\begin{theorem}\label{thm:main}
Let $w\in\Sa_N$ and $F,F'\in\Red_T(w)$. Then
\[
F\sim F'\qquad\Longleftrightarrow\qquad \langle F\rangle=\langle F'\rangle .
\]
Moreover, if $q(w)=1$ and $g=\gcd(\omega_1,\dots,\omega_c)$, then $\Red_T(w)$ has exactly
$g^{\,c-1}$ Hurwitz orbits, represented by the tuples $F(y)$ with $y$ running through a set of
representatives of $(\ZZ/g\ZZ)^{\{2,\dots,c\}}$.
\end{theorem}

\begin{proof}
The \emph{only if} direction is Lemma~\ref{lem:subgroupinv}. We
prove the \emph{if} direction.

Assume first $q(w)=1$. If $c=1$ then $m=N-1$ and there are no connectors. By Theorem~\ref{thm:normalform} all reduced reflection factorizations of $w$ are Hurwitz equivalent to the same tuple, so there is exactly one orbit, in agreement with $g^{c-1}=g^0=1$. So let $c\ge2$. Then $d_t\ne0$ for all $t$ and $g\ge1$ by Lemma~\ref{lem:nonzero}. By Theorem~\ref{thm:normalform} we may assume $F=F(y)$ and $F'=F(y')$.

\emph{Step 1:} if $y\equiv y'~\text{mod } g$ then $F(y)\sim F(y')$ by Lemma~\ref{lem:moves}.

\emph{Step 2:} We claim that $\langle F(y)\rangle$ determines $y$ modulo $g$. 

Write $W'=\langle F(y)\rangle$. By definition of the star--connector form, the directions occurring in $F(y)$ are the $N-c$ tree edges inside the cycles together with the $c-1$ connectors $\gamma_2,\dots,\gamma_c$. Hence we have $N-1$ distinct roots, and their graph is connected (each cycle is spanned by a tree, and the connectors connect all cycles to $r_1$), so it is a spanning tree of the complete graph on $\{1,\dots,N\}$. By Lemma~\ref{lem:tree} and \cite[Lemma 3.10.1]{GR01} they are a $\ZZ$-basis of $Q$ and they also generate $S_N$. In particular $\lin(W')=S_N$, so Theorem~\ref{thm:structure} yields $e\ge0$ and $\nu'\in\ZZ^N$ with
\begin{equation}\label{eq:ReflectionsFP}
T\cap W'=\{s_{\gamma,k}: \gamma \in \Phi, ~k\equiv(\nu'\mid \gamma) ~\mathrm{mod}\ e\}.
\end{equation}
Here in fact $e\ge1$ by \eqref{eq:egeq1}, since the connector pair with direction $\gamma_t$ consists of two reflections with equal direction and levels differing by $d_t\ne0$.

We claim $e=g$. First, the two entries of the connector pair with direction $\gamma_t$ lie in
$W'$ and have levels differing by $d_t$. By \eqref{eq:ReflectionsFP} their levels are congruent
modulo $e$, so $e\mid d_t$ for all $t$ and hence $e\mid g$.

For the converse divisibility we produce an explicit subgroup containing $W'$. The $N-1$
directions occurring in $F(y)$ form a $\ZZ$-basis of $Q$, so there is a unique homomorphism
\[
\psi:Q\longrightarrow\ZZ/g\ZZ
\]
sending each such direction to the level of the corresponding entry of $F(y)$ modulo $g$. This
is well defined even though the connector directions occur twice, because the two levels of the
$t$-th connector pair differ by $d_t$, and $g\mid d_t$, so they agree modulo $g$. Consider
\[
T_\psi:=\{s_{\gamma,k}\ :\ \gamma\in\Phi,\ k\equiv\psi(\gamma)\ \mathrm{mod}\ g\}.
\]
By construction all entries of $F(y)$ lie in $T_\psi$. Moreover $T_\psi$ is closed under
conjugation by its own elements, because if $k\equiv\psi(\gamma)$ and $l\equiv\psi(\alpha)$ then, by
Lemma~\ref{lem:conj}(ii), the conjugate $s_{\alpha,l}s_{\gamma,k}s_{\alpha,l}$ has direction
$s_{\alpha}(\gamma)$ and level
\[
k-(\gamma|\alpha)l\ \equiv\ \psi(\gamma)-(\gamma|\alpha)\psi(\alpha)
=\psi\big(\gamma-(\gamma|\alpha)\alpha\big) = \psi(s_{\alpha}(\gamma)) ~\text{mod } g ,
\]

Since all entries of $F(y)$ lie in $T_\psi$, we get $W'\langle F(y) \rangle\subseteq\langle T_\psi\rangle$. As in Step 3 of the proof of Theorem~\ref{thm:structure} we can choose a vector $\nu''\in\ZZ^N$ with $\psi(\gamma) \equiv (\gamma \mid \nu'')\ \text{mod }g$. We obtain
\begin{align*}
t_{\nu''}^{-1} T_{\psi}t_{\nu''} & = \{ t_{-\nu''}s_{\gamma,k}t_{\nu''}\ :\ \gamma\in\Phi,\ k\equiv\psi(\gamma)\ \mathrm{mod}\ g \}\\
& = \{ s_{\gamma,k-(\gamma \mid \nu'')}\ :\ \gamma\in\Phi,\ k\underbrace{\equiv\psi(\gamma)}_{\equiv (\gamma \mid \nu'')}\ \mathrm{mod}\ g \}\\
& = \{ s_{\gamma,l}\ :\ \gamma\in\Phi,\ g \mid l \},
\end{align*}
which generates $\Gamma_g$ by Step~4 of the proof of Theorem~\ref{thm:structure}. Hence $\langle T_\psi\rangle=t_{\nu''}\Gamma_g\,t_{\nu''}^{-1}$, and in particular the reflections of $\langle T_\psi\rangle$ are exactly the elements of $T_\psi$.

Since translations commute with one another, conjugation by $t_{\nu''}$ fixes $\mathcal T$ pointwise, so
\[
\langle T_\psi\rangle\cap\mathcal T=
t_{\nu''}\Gamma_g\,t_{\nu''}^{-1} \cap t_{\nu''}\mathcal T t_{\nu''}^{-1}= 
t_{\nu''}(\Gamma_g \cap \mathcal T)t_{\nu''}^{-1} = 
\Gamma_g\cap\mathcal T=\iota(gQ),
\]
and $W'\cap\mathcal T=\iota(eQ)$ by Theorem~\ref{thm:structure}. From $W'\subseteq\langle T_\psi\rangle$
we therefore get
\[
\iota(eQ)=W'\cap\mathcal T\ \subseteq\ \langle T_\psi\rangle\cap\mathcal T=\iota(gQ),
\]
and since $\iota$ is injective this yields $eQ\subseteq gQ$. Applying this to a root $\gamma$ gives $e\gamma\in gQ$, i.e.\ $(e/g)\gamma\in Q$ after division, which yields $g\mid e$. Together with $e\mid g$ we conclude $e=g$.

Now that $e=g$, the uniqueness statement of Theorem~\ref{thm:structure} shows that $W'$
determines the homomorphism $\gamma\mapsto(\nu'|\gamma)\bmod g$, and applying it to $\gamma_t$ gives
$(\nu'|\gamma_t)\equiv y_t\pmod g$ because $s_{\gamma_t,y_t}\in W'$. Hence $W'$ determines
$y\bmod g$, as claimed.

At this point, we are ready to show that the following statements are equivalent:
\begin{enumerate}
    \item[(I)] $F(y)\sim F(y')$
    \item[(II)] $y\equiv y'\ \mathrm{mod } ~g$
    \item[(III)] $\langle F(y)\rangle=\langle F(y') \rangle$
\end{enumerate}
(I) implies (III) by Lemma \ref{lem:subgroupinv}, (III) implies (II) by Step 2, and (II) implies (I) by Step 1 (Lemma~\ref{lem:moves}). The number of orbits is the number of residues $y\bmod g$, i.e.\ $g^{c-1}$.

\emph{The general case.} Let $w$ be arbitrary and $F,F'\in\Red_T(w)$ with $\langle F\rangle=\langle F'\rangle$. By Lemma~\ref{lem:blocksgroup} the blocks of a factorization are determined by the subgroup it generates, so $F$ and $F'$ have the same blocks $B_1,\dots,B_r$. By Lemma~\ref{lem:blocks}(ii) the subtuples $F_s,F'_s$ are reduced factorizations of the same element $w_s\in\Sa_{B_s}$, and $\langle F_s\rangle=\langle F'_s\rangle$ because both equal the subgroup of $\langle F\rangle=\langle F'\rangle$ generated by the reflections with directions inside $B_s$. By Lemma~\ref{lem:blocks}(ii) and the lower bound established in the proof of Theorem~\ref{thm:length}, $\ell_T(w)=\sum_s\ell_T(w_s)\ge N+c-2r$, while $\ell_T(w)=N+c-2q$ by Theorem~\ref{thm:length}, hence $r\ge q$, and $r\le q$ by Lemma~\ref{lem:blocks}(iii), so $r=q$. Suppose some $w_s$ admits a balanced partition with at least two parts. Refining the $s$-th block accordingly would produce a balanced partition of $w$ with more than $r=q$ parts, contradicting the maximality of $q$. Hence $w_s$ is irreducible in $\Sa_{B_s}$ and the case already treated gives $F_s\sim F'_s$ for every $s$. By Lemma~\ref{lem:blocks}(iv) we obtain $F\sim F'$.
\end{proof}

\begin{theorem}\label{thm:count}
For every $w\in\Sa_N$ the number of Hurwitz orbits on $\Red_T(w)$ is given by
\[
\sum_{\pi\in\mathcal P(w)}\ \prod_{B\in\pi}g_B^{\ \#B-1},
\]
where a part $B$ with $\#B=1$ contributes the factor $1$.
\end{theorem}

\begin{proof}
By Lemma~\ref{lem:blocks}(iii) the block partition of a reduced factorization is balanced with $r$ parts, and by the proof of Theorem~\ref{thm:length} the equality $\ell_T(w)=N+c-2q$ forces $r=q$. Hence the block partition lies in $\mathcal P(w)$. Conversely, the construction in the upper-bound part of that proof realises every $\pi\in\mathcal P(w)$ as the block partition of some reduced factorization. Two factorizations with different block partitions generate different subgroups, hence lie in different orbits.

Fix $\pi\in\mathcal P(w)$. By Lemma~\ref{lem:blocks}(iv) the orbits with block partition $\pi$ correspond to tuples of orbits, one for each part $B$, of the reduced factorizations of $w_B$ inside $\Sa_B$. Each $w_B$ is irreducible, so Theorem~\ref{thm:main} gives $g_B^{\#B-1}$ orbits for the part $B$ (note that the number of cycles of $\lin(w_B)$ is $\#B$). Multiplying over the parts and summing over $\pi$ gives the formula.
\end{proof}

\begin{corollary}\label{cor:criterion}
Let $w\in\Sa_N$. The Hurwitz action on $\Red_T(w)$ is transitive if and only if
\begin{enumerate}[label=\textup{(\alph*)}]
\item there is exactly one maximal balanced partition, i.e.\ $\#\mathcal P(w)=1$, and
\item $g_B=1$ for every part $B$ of that partition with $\#B\ge2$.
\end{enumerate}
\end{corollary}

\begin{proof}
By Theorem~\ref{thm:count} the number of orbits is a sum of $\#\mathcal P(w)$ summands, each a product of positive integers $g_B^{\#B-1}\ge1$ by Lemma~\ref{lem:nonzero}. Hence this sum equals $1$ if and only if there is exactly one summand and every factor in it equals $1$. Therefore we need $g_B^{\#B-1}=1$ which holds precisely when $\#B=1$ or $g_B=1$.
\end{proof}

\begin{remark}
Note that results similar/analogous to Theorems \ref{thm:main} and \ref{thm:count} as well as Corollary \ref{cor:criterion} have been obtained by Lewis and Wang for complex reflection groups. See \cite[Theorem~3.2, Corollary~3.4, Theorem~4.1]{LW22}.
\end{remark}

\end{document}